\documentclass{article}
\usepackage{anysize}
\marginsize{4cm}{3cm}{3cm}{3cm}

\usepackage{amssymb}
\usepackage{amsfonts}
\usepackage{amsmath}
\usepackage{latexsym}
\usepackage{amsfonts}
\usepackage{amssymb}
\usepackage{amscd}
\usepackage{theorem}
\usepackage{a4wide}

	\usepackage{hyperref}
	
	\hypersetup{colorlinks,
citecolor=red,
filecolor=blue,
linkcolor=red,
urlcolor=blue,
pdftex}

\usepackage{amsmath} 

\usepackage{graphicx}

\newcommand{\comentario}[1]{}

{\theorembodyfont{\rmfamily}
\theoremstyle{definition} \newtheorem{defn}{Definition}[section]
\theoremstyle{definition} \newtheorem{remark}[defn]{Remark}
\theoremstyle{definition} \newtheorem{eje}[defn]{Example}
}

\theoremstyle{theorem} 
  
 \newtheorem{lem}[defn]{Lemma}
 \newtheorem{thrm}[defn]{Theorem}
 
  \newtheorem{con}[defn]{Consequence}
 \newtheorem{claim}[defn]{Claim}
\theoremstyle{theorem} 
\theoremstyle{theorem} \newtheorem{prop}[defn]{Proposition}

\newenvironment{proof}{    
  \noindent
  \textbf{Proof.}}{
  \hfill $\Box$
  \vspace{3mm}
}

\newcommand{\R}{\mathbb{R}}
\newcommand{\N}{\mathbb{N}}

\newcommand{\C}{\mathbb{C}}
\newcommand{\K}{\mathbb{K}}

\newcommand{\T}{\mathfrak{T}}

\newcommand{\om}{\Omega}

\def\n{|\!|\!|}

\begin{document}

\title{\textbf{Seminar about the Bounded Approximation Property in Fr\'echet Spaces
}}
\author{\textbf{Jos\'e Bonet}}
\date{}

\maketitle

\textbf{Abstract.} The purpose of this seminar, which was presented at the Universitat Polit\`ecnica de Val\`encia in late 2012, is to explain several results concerning the bounded approximation property for Fr\'echet spaces. We give a full detailed proof of an important result due to Pe\l czy\'nski \cite{041} (see also \cite{033}) that asserts that every separable Fr\'echet space with the bounded approximation property is isomorphic to a complemented subspace of a Fr\'echet space with a Schauder basis. We also explain Vogt's example (cf. \cite{3}) of a nuclear Fr\'echet space without the bounded approximation property. This example is simpler than the original counterexample due to Dubinski. These examples solved a long standing problem of Grothendieck. Vogt \cite{1} obtained another simple example of a nuclear Fr\'echet function space without the bounded approximation property. The relation of the bounded approximation property for Fr\'echet spaces with a continuous norm and the countably normable spaces, including several results due to Dubinski and Vogt \cite{2}, is also explained.

\section{Introduction}
A topological vector space $E$ is a Fr\'echet space if it is metrizable, complete and locally convex. We use below the abbreviation ``lcs'' for ``locally convex space''. The topology of $E$ is defined by a fundamental system of seminorms $p_1 \leq p_2 \leq \ldots \leq p_n \leq \ldots $ satisfying that for each $x \in E, x \neq 0,$ there exists $n \in \N$ such that $p_n\left(x\right) > 0$. Recall that for every neighbourhood of the origin  $U \in U_0\left(E\right)$, there exist $n \in \N$ and $\varepsilon > 0$ such that $\left\{ x \in U : p_n\left(x\right) < \varepsilon \right\} \subset U$. We may assume that  a basis of neighborhoods of the origin is given  by $U_n := \left\{ x \in E : p_n\left(x\right) < 1 \right\}, \ n \in \N$. We say that $B$ is a bounded set in $E$, and we write $B \in \mathfrak{B}\left(E\right)$,  if $\sup_{ b \in B} p_n\left(b\right) < \infty$ for every $n \in \N$.

Let $\left(p_n\right)_n$ and $\left(q_m\right)_m$ be fundamental system of seminorms in $E$ and $F$ respectively. A linear operator $T: E \rightarrow F$ is continuous operator if and only if for every $m \in N$ there exists $n \in \N$ and $C > 0$ such that $q_m\left(T\left(x\right)\right) \leq p_n\left(x\right)$ for every $x \in E$. We denote $L\left(E, F\right)$ the space of linear and continuous operators from $E$ to $F$. A set $H \subset L\left(E, F\right)$ is \textit{equicontinuous} if for every $m \in \N$ there exists $n \in \N$ and $C > 0$ such that $q_m\left(T\left(x\right)\right) \leq p_n\left(x\right)$ for every $x \in E$ and for every $T \in H$. Note that this condition is equivalent to the fact that $\cap_{ T \in H} T^{-1}\left(V\right) \in U_0\left(E\right)$ for every $V \in U_0\left(F\right)$. It is also important to recall Banach-Steinhaus' theorem for Fr\'echet spaces: Let $E$ be a Fr\'echet space. $H \subset L\left(E, F\right)$ is  equicontinuous if and only if for every $x \in E$, $H\left(x\right) := \left\{ T\left(x\right) : T \in H \right\}$ is a bounded set of $F$.

\begin{defn}
We say that $E$ admits a continuous norm if there exists a norm $\left\| \cdot \right\| : E \rightarrow \R $ that is continuous for the topology of $E$; that is there exists a norm $\left\| \cdot \right\| : E \rightarrow \R $  such that there exists $n \in \N$ and $C > 0$ with $\left\|x \right\| \leq Cp_n\left(x\right)$ for every $x \in E$. If $E$ has a continuous norm, we can choose a fundamental system of seminorms $p_1 \leq p_2 \leq \ldots \leq p_n \leq \ldots $ in $E$ such that  $p_k$ is a norm for every $k \in \N$.
\end{defn}

\begin{eje}
\begin{enumerate}
\item The space $H\left( \om \right)$ with $\om \subset \mathbb{C}$ an connected open set in the complex plane endowed with the topology of uniform convergence on the compact subsets of $\om$ is a Fr\'echet space that admits a continuous norm and is not normable.
\item The space Fr\'echet $C^{\infty}\left( \left[ 0, 1 \right]\right)$ endowed by the topology given by the seminorms
             $$p_n\left(f\right) := \max_{ 1 \leq \alpha \leq n} \sup_{x \in \left[0, 1\right]} \left| f^{\left(\alpha\right)}\left( x\right)\right|,$$
     also admits a continuous norm.
\item The space $\omega := \K^{\N}$ endowed by the topology given by the seminorms
             $$p_n\left(x\right) := \max_{ 1 \leq j \leq n} \left| x_j \right| \mbox{, with } x = \left(x_j\right)_j, $$
      does not admit a continuous norm.
\item The space $C^{\infty}\left(\om\right)$ endowed by the topology given by the seminorms
              $$p_n\left(f\right) := \max_{ 1 \leq x \leq n} \sup_{x \in K_n } \left| f^{\left(\alpha\right)} \left(x \right)\right|,$$
      where $K_1 \subset K_2 \subset \ldots \subset K_n \subset \ldots$ is a fundamental sequence of compact subsets in $\om$, does not admit a continuous norm.
\end{enumerate}
\end{eje}

There are two important results concerning Fr\'echet spaces with does not have a continuous norm.

\begin{thrm}
(\textbf{Bessaga, Pe\l czy\'nski}) A Fr\'echet space does not have a continuous norm if and only if $\omega$ is isomorphic to a complemented subspace of $E$.
\end{thrm}

\begin{thrm}
(\textbf{Eidelheit}) If $E$ is a Fr\'echet space that is not normable then $E$ have a isomorphic quotient in $\omega$.
\end{thrm}

\begin{eje}
Here is a concrete example of a not normable Fr\'echet space with a quotient isomorphic to $\omega$: Consider the Fr\'echet space $H\left(\C\right)$ of entire functions endowed with the compact open topology. Select a sequence $\left(z_n\right)$ in $\C$ such that $|z_{n+1}| > |z_n|$ for each $n \in \N$ and $\lim_{n \rightarrow \infty} |z_n|= \infty$. The linear map $T: H\left(\C\right) \rightarrow \omega$ defined by  $f \mapsto f\left(z_n\right)$ is surjective by Weierstrass interpolation Theorem. The map $T$ is clearly continuous and it is open by the open mapping theorem for Fr\'echet spaces.
\end{eje}

\begin{defn}
A lcs $E$ has the \textit{bounded approximation property} (BAP) if there exists an equicontinuous net $\left(A_j\right)_{j \in J} \subset L\left(E\right)$  with dim$\left(A_j\left(E\right)\right) < \infty$ for every $j \in J$ and $\lim_{ j \in J } A_j\left(x\right) = x$ for every $x \in E$. In other words, the net $\left(A_j\right)_{j \in J}$ converges to the identity in the space $L_s\left(E\right)$, i.e.\ for the topology of pointwise or simple convergence.
\end{defn}

\begin{remark}
Let $H \subset L\left(E, F\right)$ be equicontinuous.  If $N \subset E$ is a total subset of $E$ (i.e. $\overline{\textrm{span}\left(N\right) } = E$), then the topologies of simple convergence on $E$ ($\T_s\left(E\right)$) and on $N$ ($\T_s\left(N\right)$) coincide in $H$ (\cite[39.4.(1)]{019}). In particular, if $E$ is separable and $F$ is metrizable, then the topology $\T_s\left(E\right)$ of simple convergence on $E$ is metrizable on every equicontinuous subset $H$ oh $L\left(E,F\right)$ (\cite[39.4.(7)]{019}).
\end{remark}

\begin{con}
If $E$ is a separable, metrizable lcs then $E$ has the BAP if and only if there exists $\left(A_n\right)_n \subset L\left(E\right)$ (a sequence) which is equicontinuous, dim$\left(A_n\left(E\right)\right) < \infty$ for every $n \in \N$ and $\lim_{n \to \infty} A_n\left(x\right) = x$ for each $x \in E$.

In case that $E$ is barrelled metrizable and separable, $E$ has BAP if and only if there exists $\left(A_n\right)_n \subset L\left(E\right)$ with dim$\left(A_n\left(E\right)\right) < \infty$ for every $n \in \N$ and $\lim_{n \to \infty} A_n\left(x\right) = x$ for each $x \in E$. This is a consequence of Banach-Steinhaus Theorem.
\end{con}

\begin{remark}
Let $H$ be an equicontinuous subset of $L\left(E,F\right)$. By \cite[39.4.(2)]{019} the topology $\T_s\left(E\right)$ and the topology $\T_c\left(E\right)$ of uniform convergence of precompact subsets of $E$ coincide on $H$.
\end{remark}

\begin{con}
If $E$ has the BAP, $A_j \to I$ with $j \in J$ uniformly on the precompact subsets of $E$. Accordingly, the BAP implies approximation property.
\end{con}

In what follows $E$ is a separable Fr\'echet space, and $ p_1 \leq p_2 \leq \ldots \leq p_k \leq p_{k+1} \leq \ldots $ is a fundamental system of seminorms in $E$.

We assume, without loss of generality, that $U_k := \left\{ x \in E : p_k\left(x\right) \leq 1 \right\}$, with $k \in \N$, form a basis of 0-neighborhoods in $E$.

In case $E$ is a Fr\'echet space with a continuous norm, we assume without loss of generality that all the elements of the fundamental system of seminorms $p_1 \leq p_2 \leq \ldots \leq p_k \leq p_{k+1} \leq \ldots $ are in fact norms.

\begin{remark}
If $E$ is a separable Fr\'echet space with a fundamental system of seminorms $\left(p_k\right)_k$ has the BAP, we can find $\left(A_n\right)_n \subset L\left(E\right)$, with dim$\left(A_n\left(E\right)\right) < \infty$ for every $n \in \N$ and $\lim_{n \to \infty} A_n\left(x\right) = x$ for each $x \in E$. By the Banach-Steinhaus Theorem, $\left(A_n\right)_n$ is equicontinuous. Therefore, for every $k \in \N$ there exists $l \geq k$ and $C_k > 0$ with $p_k\left(A_n\left(x\right)\right) = C_k p_{l\left(k\right)}\left(x\right)$ for every $x \in E$ and for every $n \in \N$.
\end{remark}

\begin{prop}
Let $E$ be a separable Fr\'echet space, then the following conditions are equivalent:
\begin{enumerate}
\item The BAP holds in $E$,
\item There exists $\left(A_n\right)_n \subset L\left(E\right) $,  with dim$\left(A_n\left(E\right)\right) < \infty$ for every $n \in \N$ and $\lim_{n \to \infty} A_n\left(x\right) = x$ for each $x \in E$,
\item There exists $\left(B_n\right)_n \subset L\left(E\right) $,  with dim$\left(B_n\left(E\right)\right) < \infty$ for every $n \in \N$ and $\sum_{n = 1}^{\infty} B_n\left(x\right) = x$ for each $x \in E$.
\end{enumerate}
\end{prop}
\begin{proof}
$(1)\Rightarrow(2)$ Since $E$ is a separable space, there exists a countable dense subset $F$ of $E$ Accordingly, the following topologies coincide on the  equicontinuous subsets of $L\left(E\right)$:
\begin{itemize}
\item Uniform convergence over the compact sets of $E$,
\item Pointwise convergence on $E$,
\item Pointwise convergence on $F$.
\end{itemize}
As the topology of pointwise convergence on $F$ is metrizable, the results holds.

$(2) \Rightarrow (1) $ Since $A_n\left(x\right)$ converges to $x$ when $n$ tends to infinity, then $\left\{ A_n\left(x\right) \right\}_n $ is bounded in $E$ for every $x \in E$. By Banach-Steinhaus' theorem, $\left\{A_n\right\}_n $ is equicontinuous.

$(2) \Rightarrow (3)$ Take $B_1 := A_1$ and $B_{n+1} := A_{n+1} - A_n$, for every $n \in \N$, to get the result.

$(3) \Rightarrow (2)$ Now set $A_{n} := B_{1} + \ldots + B_n$ for every $n \in \N$.
\end{proof}

\begin{remark}
Let $\left(p_k\right)_k $ be a fundamental system of seminorms in $E$. Define $$q_k\left(x\right):= \sup_{n \in \N} p_k \left( \sum_{ i = 1}^n B_i\left(x\right) \right),$$ for every $x \in E$ and for every $k \in \N$.

Since $x= \lim_{n \rightarrow \infty} \sum_{ i = 1}^n B_i\left(x\right)$, we have $p_k\left(x\right) = \lim_{n \to \infty} p_k \left( \sum_{i = 1}^n B_i\left(x\right)\right)$ and this implies that $ p_k\left(x\right) \leq q_k\left(x\right) $ for every $ x \in E $ and for every $k \in \N$.

Observe that $\sum_{i =1}^n B_i = A_n $ for each $n \in \N$. Hence $p_k\left(\sum_{i=1}^n B_i\left(x\right)\right) = p_k\left(A_n\left(x\right)\right) \leq C_k p_{l\left(k\right)}\left(x\right)$ for each $n \in \N$. Thus $q_k\left(x\right) \leq C_k p_{l\left(k\right)}\left(x\right)$ for every $x \in E$ and for every $k \in \N$. And the sequence of seminorms $\left(q_k\right)_k$ is a fundamental system of seminorms in $E$.
\end{remark}

\begin{defn}
We say that $\left\{ x_n \right\} \subset E$ is a Schauder basis in $E$ with coefficient functionals $\left\{ x_n' \right\}$ if:
\begin{itemize}
\item For every $k, n \in \N$, $\langle x_k', x_n \rangle = \delta_{n, k}$
\item For every $x \in E$, $ x = \sum_{ n = 1}^{\infty} \langle x_n', x \rangle x_n $, the series converging in  $E$.
\end{itemize}
\end{defn}

\begin{eje}
Some spaces with Schauder basis are K\"othe echelon spaces, and the Banach sequence spaces $\ell_p, 1 \leq p < \infty,$ and $c_0$.
\end{eje}

\begin{prop}
The following results holds:
\begin{enumerate}
\item[(1)] If $E$ is a lcs with the BAP and $F \subset E$ is complemented, then $F$ has the BAP, too.
\item[(2)] If $E$ is a barrelled lcs with a Schauder basis, then  $E$ has the BAP.
\end{enumerate}
\end{prop}

\begin{proof}
 (1) Let $\left(A_{\tau}\right)_{\tau \in T} \subset L\left(E\right)$ be an equicontinuous net such that dim$A_{\tau}\left(E\right) < \infty$ for every $\tau \in T$ and $\lim_{\tau \in T} A_{\tau}\left(x\right) = x $ for each $x \in E$. Let $F \subset E$ be complemented. Denote by $J: F \rightarrow E$  the canonical inclusion and by $P: E \rightarrow F$ the projection.
For each $\tau \in T$, define $B_{\tau} : F \rightarrow F$ by $B_{\tau} := P A_{\tau} J$. Clearly, dim$B_{\tau}\left(F\right) \leq $ dim$A_{\tau}\left(F\right) < \infty$ and for every $ q \in $ cs$\left(E\right)$ there exists $ q' \in $ cs$\left(E\right)$ such that $q\left(A_{\tau}\left(x\right)\right) \leq q'\left(x\right)$ for every $x \in E$ and for every $\tau \in E$. Moreover, as $P: E \rightarrow F$ is continuous, given $p \in $ cs$\left(E\right)$ there exists $q \in$ cs$\left(E\right)$ such that $p\left(Px\right) \leq q\left(x\right)$ for every $x \in E$. Then
\begin{equation*}\label{eq01}
        p\left(B_{\tau}\left(x\right)\right) = p\left( P A_{\tau} J\left(x\right)\right) = p\left( P A_{\tau} \left(x\right)\right) \leq q\left(A_{\tau} \left(x\right)\right) \leq q'\left(x\right)
\end{equation*}
for every $ x \in F $ and for every $ \tau \in T$. Thus $\left(B_{\tau}\right)_{\tau \in T}$ is equicontinuous in $L\left(F\right)$. Finally, for $x \in F$,
\begin{equation*}\label{eq02}
    \lim_{ \tau \in T} B_{\tau}\left(x\right) = \lim_{\tau \in T} P A_{\tau} \left(x\right) = P \left( \lim_{ \tau \in T} A_{\tau} \left(x\right) \right) = P\left(x\right) = x.
\end{equation*}

\vspace{.1cm}
\noindent
(2) Let $\left(x_n\right)_n \subset E$ be a Schauder basis with coefficient functionals $\left(x_n'\right)_n \subset E'$. That  is $\langle x_k', x_n \rangle = \delta_{k,n} $ and $x = \sum_{ n = 1}^{\infty} x_n'\left(x\right) x_n $ converges in $E$ for each $x \in E$.
Denote by $P_n: E \rightarrow E$ the map $P_n\left(x\right) := \sum_{k = 1}^n x_k'\left(x\right) x_k$, which is a continuous projection onto span$\left(x_1, \ldots, x_n \right)$. Since $E$ is barrelled, $\left(P_n\right)_n$ is equicontinuous. As $\lim_{ n \to \infty} P_n\left(x\right) = x$ for every $x \in E$, we conclude that $E$ has the BAP.
\end{proof}

\begin{thrm} (\textbf{Pe\l czy\'nski. 1971})
Every separable Fr\'echet space $E$ with the BAP is isomorphic to a complemented subspace of a Fr\'echet space $E_0$ with a Schauder basis. If $E$ has a continuous norm, $E_0$ can be chosen with a continuous norm.
\end{thrm}
\begin{proof}
Fix a fundamental sequence of seminorms, $\left| \cdot \right|_1 \leq \left| \cdot \right|_2 \leq \ldots \leq \left| \cdot \right|_k \leq \left| \cdot \right|_{k+1} \leq \ldots $ in $E$. By assumption there is $\left(A_n\right)_n \subset L\left(E\right)$, dim$\left(A_n\left(E\right)\right) < \infty$ for each $n \in \N$, such that, $A_n \neq 0$ for each $n \in \N$, and $\lim_{ n \to \infty} \sum_{ p = 1}^n A_p \left(x\right) = x$ in $E$ for every $x \in E$.

We first select another (more suitable) fundamental sequence of seminorms. Set $E_p := A_p\left(E\right)$, $p =1,2, \ldots$ and $m_p :=$dim$\left(E_p\right)$, $p=1,2,\ldots$ with $m_0 := 0$.  Since dim$\left(E_p\right) < \infty$, for each $p$ there is $k\left(p\right) \in \N$ such that $k\left( p -1 \right) < k \left( p \right)$ and $\left| \cdot \right|_{k\left(p\right)}$  is a norm in $E_p$. We set $\left\|\cdot\right\|_n := \left| \cdot\right|_{ k\left(n\right)}$. Clearly, $\left\| \cdot\right\|_n \leq \left\| \cdot\right\|_{ n +1} $ for each $n$ and $\left(\left\|\cdot\right\|_n\right)_n$ is a fundamental sequence of seminorms in $E$.
Fix $n \in \N$ and for $j < n$, set $F_j^n := \left( \textrm{ Ker } \left\| \cdot\right\|_j \right) \cap E_n$. As $\left\| \cdot \right\|_j \leq \left\|\cdot\right\|_{j +1}$ for each $j$, we have $F_{n -1}^n \subset F_{n-2}^n \subset \ldots \subset F_1^n \subset E_n$. They are all closed in $E_n$ and, since they are finite dimensional, each one is complemented in the previous one.

We select a complement in each step $F_{n -2 }^n = F_{n - 1}^n \oplus H_{ n -2}^n$, $F_{n - 3 }^n = \left( F_{n - 1}^n \oplus H_{ n -2}^n \right) \oplus H_{ n -3}^n$ and $E_n = F_1^n \oplus H_0^n$. We can write $E_n = H_0^n \oplus H_1^n \oplus \ldots \oplus H_{n-2}^n \oplus F_{n-1}^n $. Selecting a element in each component and writing the projections, each  $x \in E_n$ can be uniquely written as $x = \sum_{ l = 0}^{n -1} \sum_{ k\left( \textrm{finite} \right)} x_k^l$.

Fix a seminorm $\left\|\cdot\right\|_j $ with $j < n$ and consider each projection $x_k^l$ of $x$. If $l \geq j$, $x_k^l \in H_l^n \subset F_l^n = \left( \right. $ Ker$\left. \left\| \cdot \right\|_l \right) \cap E_n $ then $x_k^l \in $ Ker$\left(\left\|\cdot\right\|_j\right)$ (i.e. $\left\| x_k^l \right\|_j = 0 $); therefore $\left\| \sum_{l = j}^{n-1} \sum_{ k \left( \textrm{finite} \right) } x_k^l \right\|_j = 0$.
On the other hand, $\left\|\cdot\right\|_j $ is a norm on $H_0^n \oplus \ldots \oplus H_{ j -1}^n $. This implies that the projection $\sum_{r=0}^{j-1} \sum_{ k \left( \textrm{finite} \right) } x_k^r \to x_k^l$ is continuous for $0 \leq l < j$ and each $k$. So we can find $C_j > 0$ such that
    \begin{equation*}\label{eq03}
    \left\| x_k^l \right\|_j \leq C_j \left\| \sum_{ r = 0}^{ j -1} \sum_{ k \left( \textrm{finite} \right)} x_k^r  \right\|_j = C_j \left\| \sum_{ r = 0}^{n -1 } \sum_{ k \left( \textrm{finite} \right) } x_k^r \right\|_j, \mbox{ for } 0 \leq l < j \mbox{ and each } k.
    \end{equation*}

Accordingly, we have found for $E_p$ and $m_p$ a family of 1-dimensional operators $B_j^p : E_p \rightarrow E_p $ with $j = 1, \ldots, m_p$ such that
    \begin{equation*}\label{eq04}
    e = \sum_{ j = 1}^{m_p} B_j^p e \mbox{ for every } e \in E_p,
    \end{equation*}
and
    \begin{equation*}\label{eq05}
    \max_{ 1 \leq j \leq m_p } \left\| B_j^p e \right\|_k \leq R_p \left\| e \right\|_k \mbox{ for every } e \in E_p \mbox{ and for every } k = 1, \ldots, p.
    \end{equation*}
In fact, $R_p = \max\left( C_1, \ldots, C_p \right) $, selected as above.

Now, since $\lim_{ n \to \infty} \sum_{ p = 1}^n A_p\left(x\right) = x $ for every $ x \in X$, the sequence $\left( \sum_{ p = 1}^n A_p \right)_{ n = 1}^{\infty} $ is equicontinuous in $L\left(E\right)$, this means that, for every $ k \in \N$, there exists $ M_k > 0 $ and $ l\left(k\right) \geq k$ such that $ \left\| \sum_{ p = 1}^n A_p\left(x\right) \right\|_k \leq M_k \left\| x \right\|_{l \left(k\right) } $ for every $ x \in X $ and for every $ n \in \N$. This implies
   \begin{equation*}\label{eq07}
   \left\| A_n\left(x\right) \right\|_k \leq \left\| \sum_{ p = 1}^n A_p\left(x\right) \right\|_k + \left\| \sum_{ p = 1}^{n -1} A_p \left(x\right) \right\|_k \leq 2 M_k \left\| x \right\|_{ l\left(k\right)} \mbox{ for each } n \in \N \mbox{ and } x \in E.
    \end{equation*}
For each $ p \in \N $ select $ N_p \in \N $ with $ m_p R_p \leq N_p $ and set $N_0 = 0$. Set $C_i^p := N_p^{-1} B_j^p $ with $i = rm_p + j$, $r= 0, 1, \ldots, N_p - 1$ and $j = 1, \ldots, m_p $. Observe that there are $N_p m_p $ rank-1 operators.
$$ \left.
    \begin{array}{lc}
    r = 0, & \frac{1}{N_p}B_1^p \ldots \frac{1}{N_p}B_{m_p}^p \\
    r = 1, & \frac{1}{N_p}B_1^p \ldots \frac{1}{N_p}B_{m_p}^p  \\
    \ldots & \ldots  \\
    r = N_p - 1,  & \frac{1}{N_p}B_1^p \ldots \frac{1}{N_p}B_{m_p}^p
    \end{array}
\right. $$
If $e \in E_p$, we get
   \begin{equation*}\label{eq09}
   \sum_{ i = 1}^{m_p N_p} C_i^p e = \sum_{ r = 0}^{N_p - 1} \sum_{ j = 1}^{ m_p } \frac{1}{N_p} B_j^p e = \frac{ 1}{N_p} \sum_{ r = 0}^{N_p -1 } e = e \mbox{ for every } p \in \N.
    \end{equation*}

Moreover, for $k = 1, 2, \ldots, p$ , $e \in E_p$, and $1 \leq q \leq m_p N_p$, we get $r$ and $w$ with $0 \leq r \leq N_p -1$, $1 \leq w \leq m_p$ such that
   \begin{equation*}\label{eq10}
   \sum_{ i = 1}^q C_i^p = r \sum_{ j = 1}^{m_p} N_p^{-1} B_j^p + \sum_{ j = 1}^w N_p^{-1} B_j^p.
    \end{equation*}
Thus, for $k = 1, 2, \ldots, p $ we have
   \begin{eqnarray*}\label{eq11}
   \left\| \sum_{ j = 1}^q C_i^p e \right\|_k & \leq & \frac{ r}{N_p} \left\| \sum_{ j = 1}^{m_p} B_j^p e \right\|_k + \frac{ 1}{N_p} \left\| \sum_{ j = 1}^w B_j^p e\right\|_k \leq  \frac{ r}{N_p} \left\| e \right\|_k `\frac{1}{N_p} \sum_{ j = 1}^w \left\| B_j^p e \right\|_k  \leq \nonumber \\
   & \leq & \left\| e \right\|_k + \frac{ 1}{N_p} \sum_{ j = 1}^w R_p \left\| e\right\|_k \leq \left( 1 + \frac{ w R_p }{ N_p } \right) \left\| e \right\|_k \leq 2\left\| e \right\|_k, \nonumber
    \end{eqnarray*}
where $\frac{ w R_p }{N_p } \leq 1$ since $1 \leq w \leq m_p$ and $m_p R_p \leq N_p $. We then obtain
   \begin{equation*}\label{eq12}
   \max_{ 1 \leq q \leq m_p N_p } \left\| \sum_{ i =1 }^q C_i^p e\right\|_k \leq 2 \left\| e \right\|_k \mbox{ for every } e \in E_p \mbox{ and } k= 1, \ldots, p.
    \end{equation*}

Define now $\widetilde{A_s} := C_i^p A_p $ for $s = m_0 N_0 + \ldots + m_{p - 1} N_{ p -1} + i$, $p = 1, 2, \ldots$ and $i = 1, 2, \ldots, m_p N_p$. Observe that $ \widetilde{A_s} \in L\left(E\right)$ since $ E \overset{A_p}{\rightarrow} E_p \overset{C_i^p }{\rightarrow} E_p \hookrightarrow E$.

\begin{claim}
$\left( \sum_{ s = 1}^n \widetilde{A_s} \right)_{ n = 1}^{\infty} $ is equicontinuous in $L\left(E\right)$.
\end{claim}

            If $n \geq m_1 N_1 $ there are $t$, $q$ with $ 1 \leq q \leq m_{ t +1} N_{t + 1}$ such that
               \begin{equation*}\label{eq13}
               \sum_{ s = 1}^n \widetilde{A_s} = \sum_{ p = 1}^t \sum_{ i = 1}^{m_p N_p} C_i^p A_p\left(x\right) + \sum_{ i=1}^q C_i^{t+1} A_{t+1}.
                \end{equation*}

            Fix $k \in \N$; for $x \in E$ we have, if $k \leq t+1$,
            \begin{eqnarray*}\label{eq14}
               \left\| \sum_{ s = 1}^n \widetilde{A_s}\left(x\right) \right\|_k & \leq & \left\| \sum_{ p = 1}^t \sum_{ i = 1}^{m_p N_p} C_i^p A_p\left(x\right)\right\|_k + \left\| \sum_{ i = 1}^q C_i^{t + 1} A_{t + 1}\left(x\right) \right\|_k = \nonumber \\
               & = & \left\| \sum_{ p = 1}^t A_p\left(x\right) \right\|_k + 2 \left\|A_{ t+1}\left(x\right) \right\|_k \leq \nonumber \\
               & \leq & M_k \left\| x \right\|_{l \left( k \right) } + 4M_k \left\|x\right\|_{l \left(k\right)} = 5M_k \left\| x \right\|_{l\left(k\right)}.
               \end{eqnarray*}

            And the claim follows, since this estimates holds for all $n$ such that $k \leq t + 1$, hence for all except a finite number. Consequently, we have
            \begin{equation*}\label{eq15}
            \forall k \in \N, \exists \omega\left(k\right), K_k > 0 : \sup_{ n } \left\| \sum_{ s = 1}^n \widetilde{A_s}\left(x\right) \right\|_k \leq K_k \left\| x \right\|_{\omega\left(k\right)} \mbox{ for every } x \in E.
            \end{equation*}

\begin{claim}
$\lim_{ n \to \infty} \sum_{ s = 1}^n \widetilde{A_s} \left(x\right) = x $ for every $ x \in E$.
\end{claim}

            First, select $t \in \N$ and $q$ with $1 \leq q \leq m_{ t +1} N_{t + 1} $, for $n \geq m_1 N_1$ then
            \begin{eqnarray*}\label{eq16}
            \left\| \sum_{ s = 1}^n \widetilde{A_s}\left(x\right) - x \right\|_k & \leq & \left\| \sum_{ p = 1}^t \sum_{ i = 1}^{ m_p N_p } C_i^p A_p\left(x\right) - x \right\|_k + \left\| \sum_{ i = 1}^q C_i^{t+1} A_{t + 1}\left(x\right) \right\|_k \leq \nonumber \\
            & \leq & \left\| \sum_{ p = 1}^t A_p\left(x\right) - x \right\|_k + 2\left\| A_{t+1}\left(x\right) \right\|_k \mbox{ if } k \leq t + 1.
            \end{eqnarray*}

            Now,
            \begin{equation*}\label{eq17}
            A_{ t +1}\left(x\right) = \left( \sum_{ r = 1}^{t + 1} A_r\left(x\right) - x \right) - \left( \sum_{ r = 1}^t A_r\left(x\right) - x \right),
            \end{equation*}
            where the two expressions tends to 0 as $t$ tends to infinity. Then, $\lim_{ t \to \infty} \left\| A_{ t + 1}\left(x\right) \right\|_k = 0$. As $ n $ tends to infinity, then $t$ tends also to infinity, therefore, there exists
            \begin{equation*}\label{eq18}
            \lim_{ n \to \infty} \left\| \sum_{ s = 1 }^n A_s\left(x\right) - x \right\|_k = \lim_{ t \to \infty} \left\| \sum_{ p = 1}^t A_p\left(x\right) - x \right\|_k + \lim_{ t \to \infty} \left\| A_t \left(x\right) \right\|_k = 0
            \end{equation*}

Denote by $E_0 := \left\{ y = \left(y\left(s\right)\right)_{s \in \N} : y\left(s\right) \in \widetilde{A_s}\left(E\right)  \textrm{ and } \sum_{ s = 1}^{\infty}y\left(s\right) \textrm{ converges in } E \right\}$, endowed with the fundamental system of seminorms
\begin{equation*}\label{eq19}
\n \left(y\left(s\right)\right)_s \n_k := \sup_n \left\| \sum_{ s = 1}^n y\left(s\right) \right\|, y = \left( y \left(s\right)\right)_s \in E_0.
\end{equation*}
It is not difficult to prove that $E_0$ is a Fr\'echet space. We prove that $E_0$ has a Schauder basis.

Since dim$\left(\widetilde{A_s}\left(E\right)\right) = 1$ for each $s \in \N$, we choose $y_s \in \widetilde{A_s}\left(E\right)$, $y_s \neq 0$, for each $s \in \N$. For each $y \in \widetilde{A_s}\left(E\right)$ there is $c \in \K$ such that $y = c y_s$.
Given $s \in \N$, define $e_s := \left(\widehat{y}\left(t\right)\right)_{t \in \N}$ by $\widehat{y}\left(t\right) = 0$ if $t \neq s$ and $\widehat{y}_s\left(s\right) = y_s$. It is easy to see that $\overline{\textrm{span}\left(e_s, s \in \N\right)} = E_0$. Moreover, $ \n  \sum_{ s = 1}^n c_s e_s \n_k \leq \n \sum_{s=1}^{n+1} c_s e_s \n_{k}$, for each $n, k \in \N$ and each $c_1, \ldots, c_{n+1} \in \K$. Then Theorem 14.3.6 in \cite[p. 298]{006} implies that $\left\{ e_s \right\}_{s \in \N} $ is a Schauder basis of $X_0$.

Now define $I: E \rightarrow E_0$ by $I\left(x\right) := \left( \widetilde{A_s}\left(x\right)\right)_{s \in \N}$. Since $ x = \sum_{s = 1}^{\infty} \widetilde{A_s}\left(x\right)$ in $E$, it follows that $I\left(x\right)$ is well-defined. Moreover, $I$ and $I^{-1}: I\left(E\right) \rightarrow E$ are continuous by the estimates $\n I\left(x\right) \n_k \leq K_k \left\| x \right\|_{w_k}$ (that were proved when we showed the equicontinuity of $\left( \sum_{s = 1}^n \widetilde{A_s}\right)_n$ ) and
\begin{equation*}\label{eq20}
\left\| x \right\|_k = \left\| \lim_{n \to \infty} \sum_{s = 1}^n \widetilde{A_s}\left(x\right) \right\|_k \leq \sup_n \left\| \sum_{s = 1}^n \widetilde{A_s}\left(x\right) \right\|_k = \n I\left(x\right) \n_k.
\end{equation*}
Observe that $I^{-1}\left(\left(y\left(s\right)\right)_s\right) = \sum_s y\left(s\right)$, if $\left(y\left(s\right)\right)_s \in I\left(E\right)$.

Finally, we define a projection $L: E_0 \rightarrow I\left(E\right)$ by $L\left(\left(y\left(s\right)\right)_s\right) := \left( \widetilde{A_s}\left( \sum_{ t = 1}^{\infty} y\left(t\right) \right) \right)_s$. To check that $L$ is continuous, if $\sum_{ s = 1}^{\infty} y\left(s\right)$ converges in $E$, using
\begin{equation*}\label{eq21}
\left\| \sum_{ s = 1}^{\infty} y\left(s\right) \right\|_l = \lim_{n \to \infty} \left\| \sum_{ s = 1}^n y\left(s\right) \right\|_l \leq \sup_n \left\| \sum_{s = 1}^n y\left(s\right) \right\|_l = \n \left(y\left(s\right)\right)_s \n_l,
\end{equation*}
for each $l \in \N$ and for each $y \in E_0$, then
\begin{equation*}\label{eq22}
\n \left( A_s \left( \sum_{ t = 1}^{\infty} y\left(t\right)\right)\right)_s \n_k = \sup_n \left\| \sum_{s=1}^{n} \widetilde{A_s}\left( \sum_{ t = 1}^{\infty} y\left(t\right) \right) \right\|_k \leq K_k \left\| \sum_{ t = 1}^{\infty} y\left(t\right) \right\|_{\omega_k} \leq K_k \n \left(y\left(s\right)\right)_s \n_{ \omega_k }.
\end{equation*}
Finally, $L^2 = L$, since $\sum_{ n = 1}^{\infty} \widetilde{A_n} \left(z\right) = z$ for every $z \in E$.
\end{proof}

\section{Extension of injective maps. Vogt's Example of a nuclear Fr\'echet space without the BAP}

In this section we present Vogt's counterexample \cite{3} of a nuclear Fr\'echet space which does not satisfy the bounded approximation property. Some results on the extension of injective continuous linear maps between normed spaces are needed first.

\vspace{.3cm}

Let $E, F$ be two normed spaces and let $T \in L\left(E, F\right)$. We denote $\widehat{E}, \widehat{F}$ the completion of $E, F$ respectively. We know that $T: E \rightarrow \widehat{F}$ is a continuous map. There exists a unique continuous linear map $\widehat{T} : \widehat{E} \rightarrow \widehat{F}$  such that the restriction $\left.\widehat{T}\right|E$ of $\widehat{T}$ to $E$ coincides  $T$. It is defined as $\widehat{T}\left(x\right) := \lim_{j \to \infty} T\left(x_j\right)$, with $\left(x_j\right)_j \subset E$ and $x_j \to x$ in $\widehat{E}$ as $j \to \infty$. In general, $\widehat{T}$ need not to be injective.

\begin{eje}
Let $\left( X, \left\| \cdot \right\| \right)$ be an infinite dimensional Banach space. Take $u \in X^{\ast} \backslash X'$ ( i.e. $u:X \rightarrow \K$ is a non-continuous linear form ), and define $\n x \n := \left| u\left(x\right) \right| + \left\| x \right\|$ (observe that $\left\|\cdot\right\| \leq \n \cdot \n$ in $X$). Clearly the identity $T: \left( X, \n \cdot \n \right) \rightarrow \left(X, \left\| \cdot \right\| \right)$ is continuous and injective. Then, there exists a unique continuous linear map $\widehat{T}: \widehat{\left( X, \n \cdot \n \right)} \rightarrow \left( X, \left\| \cdot \right\| \right)$ such that $\left. \widehat{T} \right|X = T$. Clearly $T$ is surjective since $\widehat{T}\left(X\right) = T\left(X\right) = X \subset \widehat{T}\left(\widehat{\left(X, \n \cdot \n \right)}\right)$.
Assume that $T$ is injective. By the closed graph theorem, $\left(\widehat{T}\right)^{-1}$ would be continuous. This would imply $\n x \n = \left| U\left(x\right) \right| + \left\| x \right\|  \leq C\left|x\right|, $ for every $x \in X$, then $u \in X'$, a contradiction.
\end{eje}

\begin{prop}\label{inj}
Let $X, Y$ be normed spaces and let $A: \left(X, \left\|\cdot\right\|\right) \rightarrow \left(Y, \n \cdot \n\right)$ be a continuous injective linear operator. The unique continuous linear extension $\widehat{A}: \widehat{\left(X, \left\|\cdot\right\|\right)} \rightarrow \widehat{\left(Y, \n \cdot \n\right)}$ of $A$ is injective if and only if for every $\left(x_j\right)_j \subset X$, which is $\left\| \cdot \right\|$-Cauchy in $X$ such that $\lim_{ j \to \infty} \n Ax_j \n = 0$ we have $\lim_{j \to \infty} \left\| x_j \right\| = 0$.
\end{prop}
\begin{proof}
Let $\widehat{x} \in \widehat{\left(X, \left\| \cdot \right\|\right)}$ such that $\widehat{A}\widehat{x} = 0$ in $\widehat{\left(Y, \n \cdot \n\right)}$. Then, there exists a $\left( x_j \right)_j \subset X$, where $\left(x_j\right)_j$ is $\left\| \cdot \right\|$-Cauchy in $X$, such that $x_j \to \widehat{x}$ in $\widehat{\left(X, \left\| \cdot \right\|\right)}$. Using that $\widehat{A}$ is continuous $\widehat{A}x_j = Ax_j$ converges $\widehat{A}\widehat{x} = 0$ in $\widehat{\left(Y, \n \cdot \n\right)}$, hence $Ax_j \to 0$ in $\left(Y, \n \cdot \n \right)$ as $j \to \infty$. By assumption $x_j \to 0$ in $X$, hence $\widehat{x} = 0$. And $\widehat{A}$ is injective.

In order to show the converse, let $\left( x_j \right)_j $ be a Cauchy sequence in $X$ such that $\lim_{ j \to \infty} \n Ax_j \n = 0$. There is $ \widehat{x} \in \widehat{X}$ such that $x_j \to \widehat{x} $ in $\widehat{\left(X, \left\| \cdot \right\| \right)}$, thus $Ax_j \to \widehat{A}\widehat{x}$ in $\widehat{\left( Y, \n \cdot \n \right)}$. This implies $\widehat{A}\widehat{x} = 0$ and, since $\widehat{A}$ is injective, $\widehat{x} = 0$ . Therefore $\lim_{ j \to \infty} \left\| x_j \right\| = 0$.
\end{proof}

\begin{lem}
\textbf{(Vogt's main Lemma)} Let $E$ be a Fr\'echet space with a fundamental system of seminorms $\left( \left\| \cdot \right\|_k \right)_k$. Assume that $E$ has a continuous norm and the BAP. Then, there exists $k \left(0\right)$ such that for every $k \geq k\left(0\right)$ exists $l \geq k$ such that for every $\left(x_j \right)_j \subset E$ that is $\left\| \cdot \right\|_l$-Cauchy such that $\lim_{ j \to \infty} \left\| x_j \right\|_{ k \left(0\right)} = 0$ we have $\lim_{j \to \infty} \left\| x_j \right \|_k = 0$.
\end{lem}
\begin{proof}
Let $\left(A_{\tau}\right)_{\tau \in T}$ be an equicontinuous net in $L\left(E\right)$ such that $A_{\tau}\left(E\right)$ is finite dimensional for every $\tau \in T$ and $A_{\tau}x$ converges to $x$ for each $x \in E$.

Let $\left\| \cdot \right\|_{k \left(1\right)}$ be a norm. Select $k\left(0\right) \geq k\left(1\right)$ and $C > 0$ such that $\left\| A_{\tau} x \right\|_{k\left(1\right)} \leq C\left\|x\right\|_{k\left(0\right)}$ for each $x \in E$ and each $\tau \in T$.

Since $\left\| \cdot \right\|_{k\left(1\right)}$ is a norm and $A_{\tau}\left(E\right)$ is finite dimensional, for each $\tau \in T$ and $k \in \N$ there is $C_{ \tau, k} > 0$ such that
\begin{equation*}\label{eq23}
\left\| A_{\tau} x \right\|_k \leq C_{ \tau, k } \left\|A_{\tau}x \right\|_{k\left(1\right)} \leq C_{\tau, k }C\left\|x\right\|_{k\left(0\right)}.
\end{equation*}
Fix now $k \geq k\left(0\right)$ and select $l \geq k$ and $D > 0$ such that $ \left\| A_{\tau} x \right\|_k \leq D \left\| x \right\|_l $ for every $x \in E$ and for every $\tau \in T$. Let $\left(x_j \right)_j \subset E$ be a $\left\| \cdot \right\|_l$-Cauchy sequence such that $\lim_{j \to \infty} \left\| x_j \right\|_{k \left(0\right)} = 0$.  Given $\varepsilon > 0$, choose $j\left(0\right) \in \N$ such that $\left\| x_j - x_{j\left(0\right)} \right\|_l < \varepsilon $ if $ j \geq j\left(0\right)$. Given $x_{j\left(0\right)} \in E$, select $\tau \in T$ such that $\left\| x_{j \left(0\right)} - A_{\tau}x_{j\left(0\right)} \right\|_k < \varepsilon$. For $j \geq j\left(0\right)$, we have
            \begin{eqnarray*}\label{eq24}
            \left\| x_j \right\|_k & \leq & \left\| x_j - x_{j\left(0\right)} \right\|_k + \left\| x_{j\left(0\right)} - A_{\tau}x_{j\left(0\right)}\right\| + \left\| A_{\tau}\left(x_{j\left(0\right)} - x_j \right) \right\| + \left\|A_{\tau}x_j \right\| \leq \nonumber \\
            & \leq & \left\| x_j - x_{j\left(0\right)} \right\|_l + \varepsilon + C_{\tau, k }C\left\|x_{j\left(0\right)} - x_j \right\|_{k\left(0\right)} + C_{\tau, k } C\left\|x_j\right\|_{k\left(0\right)} \leq \nonumber \\
            & \leq & \left(2\varepsilon + C_{\tau, k} C \varepsilon\right) + C_{\tau, k } C \left\| x_j \right\|_{k\left(0\right)}.
            \end{eqnarray*}
Selecting $j\left(1\right)  > j\left(0\right)$ with $ \left\|x_j\right\|_{k\left(0\right)} < \varepsilon$, we get $\left\|x_j\right\|_k \leq \left(2 + 2C_{\tau, k} \right) \varepsilon$
for all $j \geq j\left(1\right)$.
\end{proof}

\begin{eje} \label{vogtex}
\textbf{(Vogt's example)} Take $0 < \rho_{ \mu, \nu} \leq 1$, $\mu, \nu \in \N$, with $\lim_{ \mu \to 0} \rho_{ \mu, \nu} = 0$ for every $\nu \in \N$. Denote $x = \left( x_{ \mu, \nu}^{n} \right)_{ n, \mu, \nu} \in \K^{\N \times \N \times \N}$. For $p \in \N$, define
    \begin{equation*}
    \left\| x \right\|_p := \sum_{n, \mu, \nu \leq p} \left| x_{\mu, \nu}^n \right|p^{n + \mu + \nu} + \sum_{n, \mu, \nu > p} \left| \rho_{ \mu, \nu} x_{\mu, \nu}^n - x_{\mu, \nu}^{n+1}\right|p^{ n + \mu + \nu}.
    \end{equation*}
Set $E := \left\{ x = \left( x_{\mu, \nu}^n \right)_{n, \mu, \nu \in \N^3} : \left\| x \right\|_p < \infty \textrm{ for every } p \in \N \right\}$. It is a Fr\'echet space and
    \begin{equation*}
    \left\| x \right\|_p \leq 2 \left(\sum_{n, \mu, \nu \leq p+1} \left| x_{\mu, \nu}^n \right|p^{n + \mu + \nu} + \sum_{n, \mu, \nu > p+1} \left| \rho_{ \mu, \nu} x_{\mu, \nu}^n - x_{\mu, \nu}^{n+1}\right|p^{ n + \mu + \nu} \right) =: \left\| x \right\|_p'.
    \end{equation*}
To see this, use the inequality
    \begin{equation*}
    \sum_{n, \mu, \nu = p+1} \left| \rho_{ \mu, \nu} x_{\mu, \nu}^n - x_{\mu, \nu}^{n+1}\right|p^{ n + \mu + (p + 1)} \leq \sum_{n, \mu} \left( \left|  x_{\mu, p+1}^n \right| + \left| x_{\mu, p+1}^{n+1}\right| \right)p^{ n + \mu + (p + 1)}.
    \end{equation*}

The canonical map $ \left( E, \left\| \cdot \right\|_{p+1} \right) \mapsto \left( E, \left\| \cdot \right\|_p' \right)$ is nuclear, then $E$ is a nuclear space. Indeed,
    \begin{equation*}
    x = \sum_{ n, \mu, \nu} (e_{n, \mu, \nu} \otimes u_{n, \mu, \nu})(x),
    \end{equation*}
where $e_{n, \mu, \nu}$ are the canonical unit vectors in $E \subset \K^{\N^3}$ and $u_{n, \mu, \nu}$ are the canonical unit vectors in the dual. Exactly in the same positions we obtain
$$
    \left.
	\begin{array}{ll}
	 \left\| e_{n, \mu, \nu} \right\|_p' & = \left\{
	\begin{array}{l}
		p^{n + \mu + \nu}, \textrm{ or } \\
        \rho_{\mu, \nu} p^{n + \mu + \nu}
    \end{array}\right. \\
    \left\| u_{n, \mu, \nu} \right\|_{p+1} & =  \left\{
	\begin{array}{l}
		\frac{ 1}{\left(p + 1\right)^{n + \mu + \nu}}, \textrm{ or } \\
        \frac{1}{ \rho_{\mu, \nu} \left(p+1\right)^{n + \mu + \nu}}
    \end{array}\right.
    \end{array}\right.
\textrm{ then } \sum_{n, \mu, \nu} \left\| e_{n, \mu, \nu} \right\|_p' \left\| u_{n, \mu, \nu}\right\|_{p+1} < \infty. $$

Now, observe that $\left\| \cdot \right\|_p$ is a norm in $E$ for all $p$. Indeed, assume $x = \left( x_{\mu, \nu}^n \right) \in E$ satisfies $\left\| x \right\|_p = 0$ then $x_{ \mu, \nu}^n = 0$ for all $n, \mu$ if $\nu \leq p$. In that case, $\rho_{ \mu, \nu} x_{\mu, \nu}^{n+1}$ with $\nu > p$ for every $n, \mu$ then $x_{\mu, \nu}^n = \rho_{\mu, \nu}^n x_{\mu, \nu}^{n+1}$ with $\nu > p$ for every $ \mu \in \N$.
Suppose there are $\mu \in \N$ and $\nu > p$ such that $x_{\mu, \nu}^1 \neq 0$. Select $q \in \N$ with $q \rho_{\mu, \nu} > 1$ and $q > \nu$. Then, since $ x \in E$,
    \begin{equation*}
   \infty  \overset{ q \geq \nu}{ \geq}  \left\| x \right\|_q \leq \sum_{n} \left| x_{ \mu, \nu}^n \right| q^{n + \mu + \nu} = \sum_n \left(\rho_{\mu, \nu}\right)^{n-1} \left| x_{\mu, \nu}^1 \right| q^{n + \mu, \nu} = \left|x_{\mu, \nu}^1\right| \sum_n \left(\rho_{\mu, \nu} q\right)^{n-1} q^{\mu + \nu + 1} = \infty,
    \end{equation*}
and this is a contradiction. Then $x_{\mu, \nu}^1 = 0$ for each $\mu$ and each $\nu > p$ implies that $x_{\mu, \nu}^n = 0$ for every $\mu, n \in \N$ and for every $\nu > p$. Therefore $ x = 0$.

To prove  that $E$ does not have the BAP, we use Vogt's main lemma and we will prove for every $p_0$ and for every $q \geq p = p_0 + 1$ there exists $\left(x_m\right)_m \subset E$, which is $\left\| \cdot \right\|_q$-Cauchy, with $\left\| x_m \right\|_{p_0}$ converging 0 as $m$ tends to infinity, but $\left\| x_m \right\|_p$ does not converge to 0. To prove this fact, given $q \geq p_0 +1$, select $\mu \in \N$ such that $\rho_{\mu, p } q < 1$. Define $x_m := \sum_{ n = 1}^m \left(\rho_{\mu, p}\right)^n e_{ n, \mu, p}$, where $e_{n, \mu, \nu}$ are the canonical unit vectors in $E$. For $l < m$ we get
    \begin{equation*}
    \left\| x_m - x_l \right\|_q = \left\| \sum_{ n = l  + 1}^{m} \left( \rho_{\mu, p} \right)^n e_{n, \mu, p} \right\|_q  \overset{q \geq p}{=} \sum_{l +1}^m \rho_{\mu, p}^n q^{n + \mu + p} = q^{\mu + p} \sum_{ l + 1}^m \left( \rho_{\mu, p} q\right)^n .
    \end{equation*}
And $\left(x_m\right)_m$ is $\left\| \cdot \right\|_q$-Cauchy, since $\rho_{\mu, p} q < 1$. On the other hand,
    \begin{equation*}
    \left\| x_m \right\|_{p_0} = \sum_{ n = 1}^m \left| \rho_{\mu, p} x_{\mu, p}^n - x_{\mu, p}^{n+1} \right| p_0^{n + \mu + p},
    \end{equation*}
where, for $n= 1, \ldots, m-1$, $\rho_{\mu, p} x_{\mu, p}^n - x_{\mu, p}^{n +1} = \rho_{\mu, p} \rho_{\mu, p}^n - \rho_{\mu,p}^{n +1} = 0$, therefore
    \begin{equation*}
    \left\| x_m \right\|_{p_0} = \rho_{\mu, p}^{m +1} p_0^{m + \mu + p} = p_0^{\mu + p - 1} \left( \rho_{\mu, p} p_0 \right)^{m + 1} \to 0 \textrm{ as } m \to \infty,
    \end{equation*}
since $\rho_{\mu, p} p_0 < \rho_{\mu, p} q < 1$.

Finally,
    \begin{equation*}
    \left\| x_m \right\|_{p} = \sum_{ n = 1}^m  \rho_{\mu, p}^{n} p^{n + \mu + p} = p^{\mu + p} \sum_{ n = 1}^m \left( \rho_{\mu, p} p \right)^{n} \geq p^{\mu + p + 1}\rho_{\mu, p} \textrm{ for every } m \in \N.
    \end{equation*}
Therefore $\left\| x_m \right\|_p$ does not converge to 0 as $m$ tends to infinity.
\end{eje}

\section{Countably normable Fr\'echet Spaces}

\begin{defn}
A Fr\'echet space $E$ is \textit{countably normable (or countably normed)} if there exists a fundamental sequence of norms $ \left( \left\| \cdot \right\|_k \right)_k $ defining the topology of $E$ such that the inclusions $i_k : \left( E, \left\| \cdot \right\|_{ k +1} \right) \rightarrow \left( E, \left\| \cdot \right\|_k \right) $ can be extended (uniquely) to an injection $ \varphi_k : \widehat{\left( E, \left\| \cdot \right\|_{ k +1} \right)} \rightarrow \widehat{\left( E, \left\| \cdot \right\|_k \right)}$ (i.e. $E$ is an intersection of Banach spaces).
\end{defn}

The following result is a consequence of proposition \ref{inj}.

\begin{lem}
  Let $X$ be a vector space with two norms $\left\| \cdot \right\| \leq \n \cdot \n$. The inclusion $i:\left(X, \n \cdot \n \right) \rightarrow \left( X, \left\| \cdot \right\| \right)$ extends uniquely to an injective continuous map if and only if for every $\left( x_n \right)_n \subset X$, which is $\n \cdot \n$-Cauchy, such that $\left\| x_n \right\| \to 0$ as $n$ tends to infinity then $\n x_n \n \to 0 $.
\end{lem}

\begin{remark}
A Fr\'echet space $E$ with a continuous norm is countably normable if and only if there exists a fundamental system $\left( \left\| \cdot \right\|_k \right)_k$ of norms on $E$ such that  for every $k \in \N$ there exists $j > k $ such that if $\left(x_n \right)_n \subset X$ is $\left\| \cdot \right\|_j$-Cauchy and $\lim_n \left\| x_n \right\|_k = 0$,  then $\lim_n \left\| x_n \right\|_j = 0$.

Indeed, if we suppose that $E$ is countably normable then it is enough to take $j = k + 1$. If the condition is satisfied, it is enough to pass to a subsequence.

As a consequence, if $E$ is a Fr\'echet space which is countably normable and $F \subset E$ is a closed subspace, then $F$ is a countably normable Fr\'echet space. To prove it we just take the restriction to $F$ of the norms given by the remark on $E$.
\end{remark}

\begin{prop}
Every Fr\'echet space $E$ with a Schauder basis and a continuous norm is countably normable.
Consequently, every separable Fr\'echet space with a continuous norm and the bounded approximation property is countably normable.
\end{prop}
\begin{proof}
Let $\left(x_n \right)_n$ be a Schauder basis in $E$ with coefficient functionals $\left(x_n'\right)_n$. We write
$$\left.
	\begin{array}{crl}
		A_n: & E   \longrightarrow  &  E \\
		 & x \longrightarrow & \langle x_n', x\rangle x_n.
	\end{array}
\right. $$
We have dim$\left(A_n\left(E\right)\right) = 1$, $A_n A_m = \delta_{n, m} A_n$ if $n \neq m$ and $ x = \sum_{ n = 1}^{\infty} A_n\left(x\right) = \sum_{ n = 1}^{\infty} \langle x_n', x \rangle x_n $ converging in $E$ for every $ x \in E$.

Given a fundamental sequence of norms $\left(\left\| \cdot \right\|_k \right)_{k \in \N}$ in $E$, define $\left| y \right|_k := \sup_n \left\| \sum_{ i = 1}^n A_i y \right\|_k$ for every $y \in E$ and $k \in \N$. Then $\left( \left| \cdot \right|_k \right)_k $ is a fundamental sequence of seminorms in $E$. Indeed,
    \begin{equation*}
    \left\| x \right\|_k = \lim_{ n \to \infty} \left\| \sum_{ i = 1}^n A_i \left(x\right) \right\|_k \leq \sup_n \left\| \sum_{ i = 1}^n A_i\left(y\right) \right\|_k = \left| x \right|_k,
    \end{equation*}
for every $k \in \N$ and for every $x \in E$. In particular $\left| \cdot \right|_k$ is a norm for each $k$. On the other hand, since $\left( \sum_{ i = 1}^n A_i \right)_n $ is equicontinuous in $L\left(E\right)$ by Banach-Steinhaus' Theorem, for every $k \in \N$, there exists $l\left(k\right) > k$ and $C_k > 0$ such that $\left\| \sum_{i = 1}^n A_i\left(x\right)\right\|_k \leq C_k \left\| x \right\|_{ l \left(k \right)}$ for every $n \in \N$ and for every $x \in E$. This implies that for every $k \in \N$ there exists $l\left(k\right) > k$ and $C_k > 0$ such that $\left| x \right|_k \leq C_k \left\|x\right\|_{l\left(k\right)}$.

Consider that map $A_j : \left( E, \left| \cdot \right|_k \right) \rightarrow \left( E, \left| \cdot \right|_k \right) $ between these normed spaces. It is continuous
    \begin{equation*}
    \left| A_j\left(y\right) \right|_k = \sup_n \left\| \sum_{ i = 1}^n A_i A_j \left(y\right) \right\|_k = \left\| A_j\left(y\right) \right\|_k  \leq \left\| \sum_{ i = 1}^j A_j\left(y\right) - \sum_{ i = 1}^{j - 1} A_j\left(y\right) \right\|_k \leq 2\left|y\right|_k,
    \end{equation*}
and if $j = 1$ then $\left| A_1\left(y\right) \right|_k \leq \left| y \right|_k $. Then there exists a unique continuous extension $A_j^k : \widehat{\left( E, \left| \cdot \right|_k \right)} \rightarrow \widehat{\left( E, \left| \cdot \right|_k \right)}$ such that $\left. A_j^k \right|E = A_j$ for each $j \in \N$. Moreover, since $A_j^k\left(E\right)$ is finite dimensional (in fact 1-dimensional), it is closed and $A_j^k \left[ \widehat{\left(E, \left| \cdot \right|_k \right)} \right] \subset A_j\left(E\right) = $span$\left(x_j\right)$. Since $A_i A_j = \delta_{i, j } A_j$ on $E$, by density $A_i^k A_j^k = \delta_{i, j} A_i^k$. We show now that $\left( \sum_{ i = 1}^n A_j \right)_n $ is equicontinuous on $L\left(E, \left| \cdot \right|_k \right)$. If $y \in E$, $m \in \N$,
    \begin{equation*}
    \left\| \sum_{ i = 1}^m A_i \left(\sum_{ j = 1}^n A_j \right)\left(y\right) \right\|_k \overset{ m \geq n }{ = } \left\| \sum_{ j = 1}^n A_j\left(y\right) \right\|_k \leq \left| y \right|_k,
    \end{equation*}
since
    \begin{eqnarray*}
    \left\| \sum_{ j = 1}^n A_j\left(y\right) \right\|_k & = & \sup_m \left\| \sum_{ i = 1}^m \left( \sum_{ j = 1}^n A_j \right)\left(y\right) \right\|_k = \nonumber \\
    & = & \sup_{ m \geq n } \left\| \sum_{ i = 1}^m A_i \left( \sum_{ j = 1}^m A_j \right)\left(y\right) \right\|_k + \sup_{ 1 \leq m < n} \left\| \sum_{ i = 1}^m A_i \left( \sum_{ j = 1}^m A_j \right)\left(y\right) \right\|_k \leq 2 \left| y \right|_k.
    \end{eqnarray*}
This implies that the extensions $\left(\sum_{ j =1}^n A_j^k \right)_n $ form also an equicontinuous set in $L\left(\widehat{\left(E, \left| \cdot \right|_k \right) }\right)$. Since $\sum_{ j = 1}^n A_j^k \to I$ pointwise in $E$ and $\left( \sum_{ j = 1}^n A_j^k \right)_n $ is equicontinuous in $L\left(\widehat{\left(E, \left| \cdot \right|_k \right) }\right)$ then, for every $ x \in \widehat{\left(E, \left| \cdot \right|_k \right) }$, $\sum_{ j = 1}^n A_j^k\widehat{x} \overset{ n }{ \to} \widehat{x}$ in $\widehat{\left(E, \left| \cdot \right|_k \right) }$.

We finally prove that the unique extension $\varphi_k : \widehat{ \left( E, \left| \cdot \right|_{k +1} \right) } \rightarrow \widehat{\left( E, \left| \cdot \right|_{ k}\right)  }$ of the identity $ \left(E, \left| \cdot \right|_{k +1} \right) \rightarrow \left(E, \left| \cdot \right|_k \right) $ is injective. Fix $ \widehat{y} \in \widehat{\left( E, \left| \cdot \right|_{ k +1} \right)}$ such that $\varphi_k \widehat{y} = 0$ in $\widehat{\left(E, \left| \cdot \right|_k \right)}$. We know $\widehat{y} = \sum_{ n = 1}^{\infty} A_n^{k +1} \widehat{y}$ the series converging in $\widehat{\left( E, \left| \cdot \right|_{ k +1} \right)}$ and the decomposition is unique, since $A_i^{k +1}A_j^{k +1} = \delta_{i, j} A_i^{k+1}$ if $i \neq j$. Moreover, $A_n^{ k +1}\left(\widehat{y}\right) \in E$, since $A_n\left(E\right)$ is finite dimensional in $E$, hence closed in $\widehat{\left( E, \left| \cdot \right|_{k +1} \right)}$.

Now $0 = \varphi_k\left(\widehat{y}\right) = \sum_{ n = 1}^{\infty} A_n^k\left(\varphi_k\left(\widehat{y}\right)\right)$, the series converging in $\widehat{\left(E, \left| \cdot \right|_k \right)}$. Since the decomposition is unique, we have $A_n^k\left(\varphi_k\left(\widehat{y}\right)\right) = 0$ for each $n \in \N$. We are done if we show that $A_n^{k +1}\widehat{y} = A_n^k\left(\varphi_k\left(\widehat{y}\right)\right)\left(=0\right)$, because this will imply $\widehat{y} = \sum_{ n = 1}^{\infty} A_n^{k+1} \widehat{y} = 0$.

To prove $A_n^{k+1} = A_n^k\left(\varphi_k\left(\widehat{y}\right)\right)$, select $\left(y_s\right) \subset E$ such that $y_s \to \widehat{y}$ in $\widehat{\left(E, \left| \cdot \right|_{k+1} \right)}$. Then $y_s$ converges to $\varphi_k\left(\widehat{y}\right)$ in $\widehat{\left(E, \left| \cdot \right|_k \right)}$. Now $A_n^k$ is the extension of $A_n$ and $A_n^{k+1}$ of $A_n$. Thus $A_n^{k+1}\left(\widehat{y}\right) = \lim_{ s \to \infty} A_n\left(y_s\right) = A_n^k\left(\varphi_k\left(\widehat{y}\right)\right)$.
\end{proof}

The advantage of the next characterization is that it is formulated in terms of an arbitrary fundamental sequence of seminorms.

\begin{thrm}
\textbf{(Dubinski, Vogt, 1985)} Let $E$ be a Fr\'echet space with a continuous norm. Let $\left( \left\| \cdot \right\|_k \right)$ be a increasing sequence of norms which define the topology of $E$. Denote $E_k = \widehat{\left(E, \left\| \cdot \right\|_k \right)} $ and $\varphi_k : E_{k + 1} \rightarrow E_k $ the unique extension of the identity $i: \left(E, \left\| \cdot \right\|_{k+1} \right) \rightarrow \left( E, \left\| \cdot \right\|_k \right)$. Then, the following are equivalent:
\begin{enumerate}
\item $E$ is countably normable
\item There exists $k_0 \in \N$ such that for every $k \geq k_0$ there exists $j > k$ such that if $\left(x_n\right)_n \subset E$ is $\left\| \cdot \right\|_j $-Cauchy and $\left\| x_n \right\|_{k_0} \to 0$, then $\left\|x_n\right\|_k \to 0$.
\end{enumerate}
\end{thrm}
\begin{proof}
In order to show $\left(1\right) \Rightarrow \left(2\right)$, let $\left( \left| \cdot \right|_k \right)_{k \in \N} $ be a fundamental sequence of norms in $E$ such that the extensions $\varphi_k : \widehat{\left(E, \left| \cdot \right|_{k+1} \right)} \rightarrow \widehat{\left(E, \left| \cdot \right|_k \right)} $ of the identity $i: \left(E, \left| \cdot \right|_{k+1} \right) \rightarrow \left( E, \left| \cdot \right|_k \right)$ are injective for each $k$.

Select $k_0 \in \N$ such that $\left| \cdot \right|_1 \leq C\left\| \cdot \right\|_{k_0} $ for some $C > 0$ (recall that both $\left( \left| \cdot \right|_k \right)$ and $\left( \left\| \cdot \right\|_k \right)_k $ are fundamental systems of seminorms of $E$). Fix $k \geq k_0$ and choose $k'$ such that $\left\| \cdot \right\|_k \leq D\left| \cdot \right|_{k'}$ for some $D > 0$. Now choose $j$ such that $\left| \cdot \right|_{k'} \leq E\left\| \cdot \right\|_j $.

Take $\left( x_n \right)_n \subset E$, which is $\left\| \cdot \right\|_j$-Cauchy and satisfies $\left\| x_n \right\|_{k_0} \to 0 $ as $n$ tends to infinity. Since $\left| \cdot \right|_1 \leq C \left\| \cdot \right\|_{k_0}$, we get $\left| x_n \right|_1 \to 0 $ as $n \to \infty$. Moreover $\left(x_n\right)_n$ is $\left| \cdot \right|_{k'}$-Cauchy in $E$ and the unique extension $\varphi_1 \circ \ldots \circ \varphi_{ k' - 1} : \widehat{\left( E, \left| \cdot \right|_{k'} \right)} \rightarrow \widehat{\left( E, \left| \cdot \right|_1 \right)}$ of the identity $i: \left(E, \left\| \cdot \right\|_{k'} \right) \rightarrow \left( E, \left\| \cdot \right\|_1 \right)$ is injective. By the lemma,  $\left| x_n \right|_{k'} \to 0 $ as $n \to \infty$. Since $\left\| \cdot \right\|_k \leq D\left| \cdot \right|_{k'}$, we conclude $\left\| x_n \right\|_k \overset{ n }{\to} 0$. Then, the proof of $\left(2\right)$ is complete.

Now, in order to show $\left(2\right) \Rightarrow \left(1\right)$ we first prove the following
   \begin{claim}
        Condition  $\left(2\right)$ implies that there exists $k_0 \in \N$ such that, for every element $x \in \cap_{ k = k_0 }^{\infty} \varphi_{k_0} \ldots \varphi_{k}\left(E_{k+1}\right)$, there exist $x_k  \in E_k, k \in \N,$ such that $x_{k_0} = x$ and $ x_k = \varphi_k x_{k+1} $ for every $k \geq k_0$ (i.e. $x$ belongs to $Proj_{ k \geq k_0}\left(\left(E_k\right)_k, \varphi_k: E_{k+1} \to E_k\right)$).
   \end{claim}

\begin{proof}
For each $k \geq k_0$, we choose $j_k > k$ satisfying $\left(2\right)$, we may select it satisfying $j_{k+1} > j_k$ for each $k \in \N$. Given $x \in \cap_{ k = k_0 }^{\infty} \varphi_{k_0} \ldots \varphi_{k}\left(E_{k+1}\right)$, we can find for each $k > k_0$, $y_k \in E_k$ such that $x = \varphi_{k_0} \ldots \varphi_{k-1} y_k$. By $\left(2\right)$, $\varphi_{k_0} \ldots \varphi_{k-1}$ is injective on $\varphi_k \ldots \varphi_{j_k -1 }\left(E_{j_k}\right)$. Indeed, $\varphi_{k_0} \ldots \varphi_{k-1}: E_k \rightarrow E_{k_0}$ and $\varphi_k \ldots \varphi_{ j_k -1 }: E_{j_k} \rightarrow E_k$ and $\left(\varphi_{k_0} \ldots \varphi_{k-1}\right)\left(\varphi_k \ldots \varphi_{ j_k -1 }\right): E_j \rightarrow E_{k_0}$ is the unique continuous extension of the identity $i: \left(E, \left\| \cdot \right\|_{j_k} \right) \rightarrow \left( E, \left\| \cdot \right\|_{k_0} \right)$. By $\left(2\right)$ if $\left(x_n\right)_n$ is $\left\|\cdot\right\|_{j_k}$-Cauchy and $\left\|x_n\right\|_{k_0} \to 0$, therefore $\left\| x_n \right\|_k \to 0$.

Suppose $\left(\varphi_{k_0} \ldots \varphi_{k-1}\right)\left(\varphi_k \ldots \varphi_{ j_k -1 }\right)\left(z\right) = 0$ with $z \in E_{j_k}$. Select $\left(z_n \right)_n \subset E$, which is $\left\| \cdot \right\|_{j_k}$-Cauchy with $z_n \to z$ in $E_{j_k}$. We obtain
    \begin{equation*}
    \left(\varphi_{k_0} \ldots \varphi_{k-1}\right)\left(\varphi_k \ldots \varphi_{ j_k -1 }\right)\left(z_n\right) = z_n \to \left(\varphi_{k_0} \ldots \varphi_{k-1}\right)\left(\varphi_k \ldots \varphi_{ j_k -1 }\right)\left(z\right) = 0 \textrm{ in } E_{k_0},
    \end{equation*}
then $\left\| z_n \right\|_{k_0} \overset{n}{\to} 0$; therefore, $\left\| z_n \right\|_{k} \overset{n}{\to} 0$.

On the other hand, as $z_n \overset{n}{\to} z$ in $E_{j_k}$ then $z_n = \left(\varphi_k \ldots \varphi_{ j_k -1 }\right)\left(z_n\right) = \overset{n}{\to} \left(\varphi_k \ldots \varphi_{ j_k -1 }\right)\left(z\right)$ in $E_k$. Therefore $z_n \overset{k}{\to} \varphi_k \ldots \varphi_{ j_k -1 }\left(z\right)$ in $E_k$ but $\left\|z_n\right\|_k \overset{n}{\to} 0$ and this implies $\left(\varphi_k \ldots \varphi_{ j_k -1 }\right)\left(z\right) = 0$.

We define now $x_{k_0} := x$ and $x_k := \varphi_k \ldots \varphi_{ j_k -1 }\left(y_{j_k}\right)$ for each $k \in \N$. Observe first that the injectivity of $\varphi_{k_0} \ldots \varphi_{ k -1 } $ on $ \varphi_k \ldots \varphi_{ j_k -1 }\left(E_{j_k}\right)$ implies $\varphi_k \ldots \varphi_{ j_k -1 }\left(y_{j_k}\right) = \varphi_k \ldots \varphi_{ j -1 }\left(y_j\right)$ for all $j \geq j_k$. Both belong to $\varphi_k \ldots \varphi_{ j_k -1 }\left(E_{j_k}\right) \subset E_k $ and both are mapped to $x$ by $\varphi_{k_0} \ldots \varphi_{k -1}$. In particular, for $j = j_{k +1} $ we get $\varphi_k \ldots \varphi_{ j_k -1 }\left(y_{j_k}\right) = \varphi_k \ldots \varphi_{ j_k }\left(y_{j_{k+1}}\right)$. Now $x_{k + 1} = \varphi_{k+1} \ldots \varphi_{ j_k  }\left(y_{j_{k+1}}\right)$, thus $\varphi_k x_{k+1} = \varphi_k \varphi_{k+1} \ldots \varphi_{ j_k  }\left(y_{j_{k+1}}\right) = \varphi_{k} \ldots \varphi_{ j_k - 1 }\left(y_{j_{k}}\right) = x_k$. And $x_{k_0+1} = \varphi_{k_0 +1} \ldots \varphi_{ j_{k_0 + 1} - 1 }\left(y_{j_{k_0+1}}\right)$ and $\varphi_{k_0} x_{k_{0} + 1} = \varphi_{k_0} \varphi_{k_0+1} \ldots \varphi_{ j_{k_0 +1}  }\left(y_{j_{k_0+1}}\right) = x = x_{k_0}$. And the claim is proved.
\end{proof}

We may assume that $k_0 = 1$ in the claim. So, for every $x \in \cap_{k = 1}^{\infty} \varphi_1 \ldots \varphi_k \left(E_{k+1}\right)$, there exists $\left(x_k\right)_k$ such that $x_k \in E_k$, with $x_1 = x$ and $x_k = \varphi_k x_{k+1}$ for each $k \geq 1$. Set $F_k := \varphi_1 \ldots \varphi_{ k }\left(E_{k+1}\right)$ with the quotient norm induced by $E_{k+1}$. The space $F = \cap_k F_k$ with the projective topology is a countably normed Fr\'echet space. Observe that $F \subset E_1 $ algebraically and the injection is continuous, since each map $\varphi_1 \ldots \varphi_k: E_{k+1} \rightarrow E_1$ is continuous.

We denote by $P_k : E \rightarrow E_k$ the canonical inclusion. Recall that $\left\| \cdot \right\|_1$ is a norm, hence $P_k : E \rightarrow \widehat{\left(E, \left\| \cdot \right\|_k \right)}$ is injective. We show that $P_1 E = F$ ( in $E_1$). By definition of $E$, $P_1 = \varphi_1 \ldots \varphi_k P_{k+1} $ for each $k \in \N$ then $P_1 E = \varphi_1 \ldots \varphi_k P_{k +1} E \subset \varphi_1 \ldots \varphi_k\left(E_{k+1}\right)$ for each $k$; therefore, $P_1 E = F$. On the other hand, if $y \in F \subset E_1$, $y = \cap_{k = 1}^{\infty} \varphi_1 \ldots \varphi_k\left(E_{k+1}\right)$ we apply the claim to find $\left(x_k\right)_k$, $x_k \in E_k$ for each $k$ such that $x_1 = y$ and $\varphi_k x_{k+1} = x_k$ for each $k \in \N$.  Since $E$ is a Fr\'echet space and $E = proj_k \left(E_k, \varphi_k\right)$, there is $x \in E$ with $P_k \left(x\right) = x_k$ for each $k$. In particular $P_1\left(x\right) = y$ and $F \subset P_1 E$.  Thus $P_1: E \rightarrow F \subset E_1 $ is bijective. We know that $P_1: E \rightarrow E_1 $ is continuous and the inclusion $F \subset E_1 $ is also continuous. If we prove that $P_1$ has closed graph as a map from $E$ to $F$, the closed graph theorem implies that $P_1$ is a continuous and (being bijective) by the open graph theorem an isomorphism.  Suppose $x_n \to x$ in $E$ and $P_1x_n \to y$ in $F$, then $P_1 x_n \to y$ in $E_1$ (since $F \hookrightarrow E_1$ is continuous) and $P_1 x_n \to Px$ in $E_1$ (since $P_1: E \hookrightarrow E_1 $ is continuous), therefore $E_1$ is Banach/Hausdorff and then $Px = y$.
\end{proof}

\begin{con}
Vogt's Example \ref{vogtex} is not countably normable.
\end{con}


\vspace{.6cm}

\textbf{Author's address:} \\

Jos\'e Bonet  \\

Instituto Universitario de Matem\'atica Pura y Aplicada \\

Universitat Polit\`ecnica de Val\`encia \\

E-46071 Valencia, SPAIN \\

email: jbonet@mat.upv.es

\end{document}